\pdfoutput=1
\documentclass{article}
\usepackage{maa-monthly}

\usepackage[hidelinks]{hyperref}

\theoremstyle{theorem}
\newtheorem{theorem}{Theorem}

\theoremstyle{definition}
\newtheorem*{example}{Example}

\begin{document}

\title{A Partition Identity Related to\\ Stanley's Theorem}
\markright{A Partition Identity Related to Stanley's Theorem}
\author{Mircea Merca and Maxie D. Schmidt}

\maketitle

\begin{abstract}
In this paper, we use the Lambert series generating function for Euler's totient function to introduce a new identity for the number of $1$'s in the partitions of $n$. A new expansion for Euler's partition function $p(n)$ is derived in this context. 
These surprising new results connect the famous classical totient function from multiplicative number theory to the additive theory of partitions.
\end{abstract}


\section{Introduction.}

A partition of a positive integer $n$ is a sequence of positive integers whose sum is $n$. The order of the summands is unimportant when writing the partitions of $n$, but for consistency, a partition of $n$ will be written with the summands in a nonincreasing order \cite{Andrews76}. For example, the partitions of $5$ are given as:
$$ 5,\ 4+1,\ 3+2,\ 3+1+1,\ 2+2+1,\ 2+1+1+1,\ 1+1+1+1+1.$$
As usual, $p(n)$ denotes the number of partitions of $n$ and $p(0)=1$.
Integer partitions are of particular interest in combinatorics and number theory, partly because many profound questions concerning integer partitions are easily stated, but not easily proved. There are a variety of partition identities of the form 
``\emph{the number of partitions of $n$ satisfying condition $A$ is equal to the number of partitions of $n$ satisfying condition $B$}.''

As we can see in our example, the number of $1$'s in the partitions of $5$ is equal to $12$. On the other hand, the number of different parts appearing in the partitions of $5$ is $1+2+2+2+2+2+1=12$.
The following result in partition theory has been widely attributed to Richard Stanley, although it is a particular case of a more general result that had been established by Nathan Fine fifteen years earlier \cite{Gilbert}.

\begin{theorem}[Stanley]
	The number of $1$'s in the partitions of $n$ is equal to the number of parts that appear at least once in a given partition of $n$, summed over all partitions of $n$.
\end{theorem}

In this paper, we shall prove a new identity for the number of $1$'s in the partitions of $n$. 
This new result involves a well-known object in multiplicative number theory: Euler's totient $\phi(n)$. This is a multiplicative function that counts the totatives of $n$, which are the positive integers less than or equal to $n$ that are relatively prime to $n$. 
In what follows, we denote by $S^{(r)}_{n,k}$ the number of $k$'s in the partitions of $n$ 
with the smallest part at least $r$. 

\begin{theorem}\label{T2}
	The number of $1$'s in the partitions of $n$ is equal to
	$$\sum_{k=2}^{n+1} \phi(k) S^{(2)}_{n+1,k}.$$
\end{theorem}

\begin{example}
	We have already seen above that the number of $1$'s in the partitions of $5$ is equal to $12$.
	The partitions of $6$ that do not contain $1$ as a part are:
	$$ 6 = 4+2 = 3+3 = 2+2+2 .$$
	We also have
	$$\phi(2)\cdot 4 + \phi(3) \cdot 2 + \phi(4) \cdot 1 + \phi(6) \cdot 1 = 4 + 4 + 2 + 2 = 12.$$
\end{example}

The set of partitions of $n$ that contain $1$ as a part can be obtained from the set of partitions of $n-1$ adding to each partition a single $1$.  This is a typical example of a bijection between two sets of partitions. For this reason, we can say that our theorem establishes a connection between the set of partitions of $n$ that contain $1$ as a part and the set of partitions of $n$ that do not contain $1$ as a part, i.e.,  
$$\sum_{1+t_1+2t_2+\cdots+nt_n=n} t_1 = \sum_{2t_2+\cdots+nt_n=n} \phi(2) t_2 +\cdots+\phi(n) t_n.$$

\section{Proof of Theorem \ref{T2} via generating functions.}
For all $k\geqslant r$, elementary techniques in the theory of partitions \cite{Andrews76} 
give the generating function 
\begin{align*}
\sum_{n=k}^{\infty} S^{(r)}_{n,k} q^n 
& = \left(q^k+q^{2k}+q^{3k}+\cdots  \right)\cdot \frac{(1-q)(1-q^2) \cdots (1-q^{r-1})}{(q; q)_{\infty}}\\
& = \frac{q^k}{1-q^k} \cdot \frac{1}{(q^r; q)_{\infty}},
\end{align*}
where 
$$(a;q)_\infty = \prod_{k=0}^{\infty} (1-aq^k).$$
Because the infinite product $(a;q)_\infty$ diverges when $a\neq0$ and $|q|\geqslant 1$, whenever $(a; q)_\infty$ appears in a formula, we shall assume that $|q|<1$.
Then we can write
\begin{align*}
\frac{1-q}{(q;q)_\infty}\sum_{k=1}^{\infty} \frac{\phi(k) q^k}{1-q^k} 
 = \frac{q}{(q;q)_\infty} + \sum_{n=2}^{\infty} \sum_{k=2}^{\infty}  \phi(k) S^{(2)}_{n,k} q^n.
\end{align*}
On the other hand, considering the well-known Lambert series generating function \cite[Theorem 309]{Hardy}
$$\sum_{k=1}^{\infty} \frac{\phi(k) q^k}{1-q^k} = \frac{q}{(1-q)^2},$$
we obtain
\begin{align*}
\frac{1-q}{(q;q)_\infty}\sum_{k=1}^{\infty} \frac{\phi(k) q^k}{1-q^k} 
 = \frac{q}{(q;q)_\infty} + \frac{q}{1-q} \cdot \frac{q}{(q;q)_\infty} 
 = \frac{q}{(q;q)_\infty} + \sum_{n=1}^\infty S^{(1)}_{n,1} q^{n+1}.
\end{align*}
Thus, we deduce the identity
$$S^{(1)}_{n,1} = \sum_{k=2}^{n+1} \phi(k) S^{(2)}_{n+1,k},$$
which completes the proof of the theorem.  

\section{A combinatorial proof of Theorem \ref{T2}.}
We denote by $p_r(n)$ the number of partitions of $n$ with the smallest part at least $r$. 
Since the set of partitions of $n$ that contain $k$ as a part can be obtained from the set of partitions of $n-k$ adding to each partition a single $k$, for all $k\geqslant r$ we note that $p_r(n-t\cdot k)$ counts the partition of $n$ with at least $t$ parts equal to $k$ and the smallest part at least $r$. Moreover, for all $k\geqslant r$ we 
deduce that $S^{(r)}_{n,k}$ can be expressed in terms of the partition function $p_r(n)$ as follows:
$$S^{(r)}_{n,k} = \sum\limits_{t=1}^{\lfloor n/k \rfloor} p_r(n-t\cdot k).$$
For arbitrary $\{a_n\}_{n\geqslant 1}$ and $\{b_n\}_{n\geqslant 0}$, we have the identity 
$$\sum_{k=1}^n \left( \sum_{d|k} a_d\right) b_{n-k} = \sum_{k=1}^n \left( \sum_{i=1}^{\lfloor n/k \rfloor} b_{n-i\cdot k}\right) a_k.$$
The proof follows easily replacing $a_n$ by $\phi(n)$ and $b_n$ by $p_2(n)$
in the previous equation. Hence we obtain that
\begin{align*}
\sum_{k=1}^n \left( \sum_{d|k} \phi(d)\right) p_2(n-k)
& = \sum_{k=1}^n k \big( p(n-k)-p(n-1-k) \big)\\
& = p(n-1)+ \sum_{k=1}^{n-1} p(n-k) \\
& = p(n-1)+ S^{(1)}_{n-1,1}
\end{align*}
and
\begin{align*}
\sum_{k=1}^n \left( \sum_{i=1}^{\lfloor n/k \rfloor} p_2(n-i\cdot k)\right) \phi(k)
& = \sum_{i=1}^n p_2(n-i) + \sum_{k=2}^n S^{(2)}_{n,k} \phi(k) \\
& = p(n-1)+ \sum_{k=2}^n S^{(2)}_{n,k} \phi(k),
\end{align*}
where we have invoked Euler's classical formula \cite[Theorem 63]{Hardy}
$$\sum_{d|n} \phi(d) = n,$$
and the relation
$p_2(n)=p(n)-p(n-1)$.

\section{A new formula for Euler's partition function.} 

In this section, motivated by the first proof of Theorem \ref{T2}, we provide a new formula for the number of partitions of $n$ that involves the half totient function, $\phi(n)/2$.
For $n>2$, we remark that $P_{\phi}(n)=\phi(n)/2$ counts the number of partitions of $n$ into two relatively prime parts. 

\begin{theorem}
\label{theorem_T3} 
For $n\geqslant 0$, 
$$p(n) = \sum_{k=3}^{n+3} P_{\phi}(k) S^{(3)}_{n+3,k}.$$
\end{theorem}

\begin{proof}
Considering the generating function for $S^{(3)}_{n,k}$, we can write
\begin{align*}
\frac{(1-q)(1-q^2)}{(q;q)_\infty}\sum_{k=1}^{\infty} \frac{\phi(k) q^k}{1-q^k} 
= \frac{q+q^2-2q^3}{(q;q)_\infty} + \sum_{k=3}^{\infty} \sum_{n=3}^{\infty} \phi(k) S^{(3)}_{n,k} q^n 
\end{align*}
and 
\begin{align*}
\frac{(1-q)(1-q^2)}{(q;q)_\infty}\sum_{k=1}^{\infty} \frac{\phi(k) q^k}{1-q^k} 
 = \frac{q+q^2}{(q;q)_\infty}.
\end{align*}
Hence we deduce that
$$ \sum_{n=3}^\infty \left(\sum_{k=3}^{n} \phi(k) S^{(3)}_{n,k} \right) q^n 
= \frac{2q^3}{(q;q)_\infty}=\sum_{n=3}^\infty 2p(n-3) q^{n},$$
where we have invoked the well-known generating function of $p(n)$, i.e.,
$$\sum_{n=0}^\infty p(n) q^n = \frac{1}{(q;q)_\infty}.$$
The theorem is proved.  
\end{proof}

\section{Conclusions.} 

We have connected Euler's totient function with a new decomposition of Stanley's theorem and provided two 
stylistically different proofs of our result: one following the method of generating functions and another with purely combinatorial methods. Along the way, we have derived a new formula for Euler's partition function 
$p(n)$ expanded in terms of our additive restricted partition functions $S^{(r)}_{n,k}$ when $r = 3$. 
These results are significant and special in form given their unusual new connections with the functions of multiplicative number theory and the more additive nature of the theory of partitions. 
It seems that the first result that combines Euler's totient function $\phi(n)$ and Euler's partition function $p(n)$ has been recently discovered 
in \cite[Corollary 5.13]{Merca}.
Finally, we remark that in 2005 Andrews and Perez \cite{Andrews05} presented an interesting relation which involves Euler's totient function and compositions (i.e., ordered partitions of natural numbers).

\begin{acknowledgment}{Acknowledgment.}
The authors appreciate the anonymous referees for their comments and careful reading of the original version of this paper.	
\end{acknowledgment}

\begin{biog}
\item[Mircea Merca] 
\begin{affil}
Academy of Romanian Scientists, Splaiul Independentei 54, Bucharest, 050094 Romania. \\
\href{mailto:mircea.merca@profinfo.edu.ro}{mircea.merca@profinfo.edu.ro} 
\end{affil}

\item[Maxie D. Schmidt] 
\begin{affil}
Georgia Institute of Technology, School of Mathematics, Atlanta, GA 30332 USA. \\
\href{mailto:maxieds@gmail.com}{maxieds@gmail.com}, 
\href{mailto:mschmidt34@gatech.edu}{mschmidt34@gatech.edu} 
\end{affil}
\end{biog}
\vfill\eject

\bigskip
\footnoterule
\footnotesize{http://dx.doi.org/10.XXXX/amer.math.monthly.122.XX.XXX}
\footnotesize{MSC: Primary 05A17, Secondary 11B75; 11P81}

\end{document}